\documentclass{amsart}
\title{The Homogeneous Spectrum of Milnor-Witt $K$-Theory}
\author{Riley Thornton}
\address{Reed College}
\email{rithorn@reed.edu}

\usepackage{amsmath, amssymb, amsthm, graphics, color}
\usepackage[backend=bibtex]{biblatex}
\usepackage[all]{xy}

\bibliography{kgroup}

 \subjclass[2010]{14F42, 19G12, 18E30}
\keywords{Milnor-Witt $K$-theory, tensor triangular geometry, stable motivic homotopy theory}

\newtheorem{thm}{Theorem}[section]

\newtheorem{lem}[thm]{Lemma}
\newtheorem{cor}[thm]{Corollary}

\theoremstyle{definition}
\newtheorem{dfn}[thm]{Definition}
\theoremstyle{remark}

\newcommand{\km}[1]{K^{M}_{#1}}
\newcommand{\kmw}[1]{K^{MW}_{#1}}

\newcommand{\N}{\mathbb{N}}
\newcommand{\Z}{\mathbb{Z}}

\newcommand{\inv}{^{-1}}

\newcommand{\ip}[1]{\langle {#1} \rangle}
\newcommand{\pf}[1]{\langle\langle{#1}\rangle\rangle}

\newcommand{\chr}{\operatorname{char}}
\newcommand{\spec}{\operatorname{Spec}}
\newcommand{\mspc}{\operatorname{MinSpec}^h}
\newcommand{\sgn}{\operatorname{sgn}}

\newcommand{\SHmot}{\operatorname{SH}^{\mathbb{A}^1}}
\newcommand{\Spc}{\operatorname{Spc}}

\begin{document}

\begin{abstract}
For any field $F$ (of characteristic not equal to 2), we determine the Zariski spectrum of homogeneous prime ideals in $\kmw*(F)$, the Milnor-Witt $K$-theory ring of $F$.  As a corollary, we recover Lorenz and Leicht's classical result on prime ideals in the Witt ring of $F$.  Our computation can be seen as a first step in Balmer's program for studying the tensor triangular geometry of the stable motivic homotopy category.
\end{abstract}

\maketitle

\section{Introduction}\label{sec:intro}

In this note we completely determine the Zariski spectrum of homogeneous prime ideals in $\kmw*(F)$, the Milnor-Witt $K$-theory of a field $F$.  This graded ring contains information related to quadratic forms over $F$ --- in fact, $\kmw0(F)\cong GW(F)$, the Grothendieck-Witt ring of $F$ --- and the Milnor $K$-theory of $F$, which appears as a natural quotient of $\kmw*(F)$.  While the prime ideals in $GW(F)$ are known classically via a theorem of Lorenz and Leicht \cite{LL} (see also \cite[Remark 10.2]{Ba}), we discover a more refined structure in $\spec^h(\kmw*(F))$, including a novel class of characteristic $2$ primes indexed by the orderings on $F$ which all collapse to the fundamental ideal $I\subseteq GW(F)$ in degree $0$.

Much of the interest in $\kmw*(F)$ stems from the distinguished role it plays in Voevodsky's stable motivic homotopy category, $\SHmot(F)$.  Indeed, a theorem of Morel \cite[\textsection 6, p 251]{Mo2} identifies $\kmw*(F)$ with a graded ring of endomorphisms of the unit object in $\SHmot(F)$.  As $\SHmot(F)$ is a tensor triangulated category (with tensor given by smash product, $\wedge$), it may be studied via Balmer's methods of tensor triangular geometry \cite{Ba}. More specifically, we can look at the full subcategory of compact objects, $\SHmot(F)^c$. In this context, the goal is to determine the structure of the triangular spectrum $\Spc(\SHmot(F)^c)$ of thick subcategories of $\SHmot(F)^c$ which satisfy a ``prime ideal'' condition with respect to $\wedge$.  Balmer's primary tool in the study of triangular spectra is a naturally defined continuous map
\[
  \rho^\bullet:\Spc(\SHmot(F))\to \spec^h(\kmw*(F))
\]
with codomain the Zariski spectrum of homogeneous prime ideals in $\kmw*(F)$.

By identifying $\spec^h(\kmw*(F))$, we undertake a first step in Balmer's program for studying the tensor triangular geometry of $\SHmot(F)$.  In particular, this raises the possibility of studying surjectivity properties of $\rho^\bullet$ (which, in general, are unknown --- see \cite[Remark 10.5]{Ba}) by explicitly constructing triangular primes lying over points in $\spec^h(\kmw*(F))$.

\subsection*{Acknowledgements}

I thank Kyle Ormsby for advising this research and editing this write up, as well as for many helpful conversations. I thank Aravind Asok for helpful comments on the pullback presentation of Milnor-Witt K-theory, Bob Bruner for pointing out that $L^*$ is essentially a Rees algebra, and the referee for helpful comments. This research was conducted with support under NSF grant DMS-1406327.

\subsection*{Outline of the paper}

The determination of $\spec^h(\kmw*(F))$ proceeds as follows. Section \ref{sec:kmw} gives general background on Milnor-Witt $K$-theory and states our main result. In subsections \ref{subsec:kmw/h} and \ref{subsec:kmw/eta}, the homogeneous spectra of two quotients of $\kmw*(F)$ are determined, and in subsection \ref{subsec:spec} the two quotient spectra are stitched together to get the full spectrum.

\section{Milnor-Witt $K$-Theory}\label{sec:kmw}

The Milnor-Witt $K$-theory of a field, $\kmw*(F)$, is a graded ring associated to a field by taking a certain quotient of the free algebra on a symbol $\eta$ and the set of formally bracketed units in the field as follows:

\begin{dfn}
For a set $S$, $[S]=\{[s]: s\in S\}$ is the set of (purely formal) symbols in $S$.

The free associative algebra on $[F^\times]\cup\{\eta\}$ is
\[ FrAl({[F^\times]\cup\{\eta\}})=\left\{\sum_{1\leq i\leq n} a_i \sigma_{i1},...,\sigma_{ij_i}: \;a_i\in\Z,\; \sigma_{ij}\in [F^\times]\cup\{\eta\},\; n\in\N\right\} \]
with multiplication and addition completely determined by the ring axioms.

Milnor-Witt $K$-theory, $\kmw*(F)$, is the quotient $FrAl([F^\times]\cup\{\eta\})/K$, where $K$ is the ideal generated by
\begin{enumerate}
\item $[ab]-[a]-[b]-\eta[a][b]$ (twisted logarithm),
\item $[a][b]$ where $a+b=1$ (Steinberg relation),
\item $[a]\eta-\eta[a]$ (commutativity), and
\item $(2+[-1]\eta)\eta$ (Witt relation).
\end{enumerate}

We put a grading $\kmw*(F)$ by declaring $\eta$ to be of degree $-1$ and $[a]$ be of degree $1$ for $a\in F^\times$.
\end{dfn}

We will usually suppress the field $F$ in our notation. Note that $\kmw*/(\eta)$ is simply Milnor $K$-theory. From this point on, we will assume that $\operatorname{char}(F)\not=2$. This assumption is needed for Lemmas \ref{lem:fibreproduct} and \ref{lem:lstar} and to make sure the orders defined in this paper are well-defined.

 Morel \cite{Mo} gives a presentation of $\kmw*$ as a pullback of Milnor $K$-theory and a ring~$L^*$ associated to the Witt ring. A correct proof of this presentation in every non-$2$ characteristic can be found in \cite{GSZ} We define $L^*$ and give Morel's pullback presentation below, though we will only make use of this presentation obliquely.

\begin{dfn} The ring $L^*$ is defined by
\[L^*:=\bigoplus_{n\in \Z} I^n\eta^{-n}\subseteq W(F)[\eta,\eta^{-1}]\]
where $W(F)$ is the Witt ring of $F$, $I$ is the fundamental ideal in $W(F)$, and $I^n=W(F)$ for $n\leq 0$ by fiat.
\end{dfn}

Note that $\eta$ has degree $-1$ in $L^*$, and $\eta\inv\not\in L^*$. Note also the similarity between $L^*$ and the Rees algebra of $I$ in $W(F)$; indeed, $L^*$ is essentially a $\Z$-graded version of the Rees algebra, sometimes called an extended Rees algebra. The ring $L^*$ is elsewhere called $I^*$. We adopt this nonstandard notation to avoid confusion between $L^1$ and $I\subseteq L^0$.

The quotient of $L^*$ by $(I)$ is simply the graded ring associated with the $I$-adic filtration of the Witt ring,\footnote{Note that $(I)$ is the ideal generated by the copy of $I$ in $L^0$.} $Gr_I=\bigoplus_{n\in \N} I^n/I^{n+1}$, which is, by the now settled Milnor conjecture \cite{Vo, Mo3}, isomorphic to $\km*/(2)$ via the Milnor map. 

\begin{lem}[Morel, \cite{Mo}; Gille-Scully-Zhong, \cite{GSZ}]\label{lem:fibreproduct}
The Milnor-Witt K-theory ring, $\kmw*$, is isomorphic to the pullback in the diagram below.
\[
\xymatrix{
L^*\times_{Gr_I} \km* \ar[r] \ar[d] & \km* \ar[d]_{q'\circ m} \\
L^*\ar[r]_q & Gr_I
}
\]
where $q$ and $q'$ are the quotients by $(I)$ and $(2)$, respectively, and $m$ is the Milnor map.
\end{lem}

The ring $\kmw*$ is an $\epsilon$-commutative graded ring,\footnote{Recall that a graded ring $R_*$ is $\epsilon$-commutative when $ab=\epsilon^{m+n}ba$ when $a\in R_m$ and $b\in R_n$.} with $\epsilon=-(1+\eta[-1])$. We may then study its homogeneous spectrum, or the set of prime ideals generated by homogeneous elements:
\[\spec^h(\kmw*)=\{J\in \spec(\kmw*): J=(J\cap \kmw n :n\in \Z)\}.\]
We equip $\spec^h(\kmw*)$ with the Zariski topology generated by the subbasis elements
\[D(q):=\{J: q\not\in J\}\]
as $q$ ranges over $\kmw*$. The main theorem of this paper is a complete characterization of $\spec^h(\kmw*)$ in terms of the orderings\footnote{Recall that an ordering $\alpha$ on a field is uniquely determined by its positive cone $P_\alpha$, and may be viewed as a group epimorphism $\alpha:F^\times\rightarrow \{\pm 1\}$ satisfying additivity: $\alpha(a+b)=1$ if $a,b$ are positive, \emph{i.e.}, when $\alpha(a)=\alpha(b)=1$.} on $F$.

\begin{thm}\label{maintheorem} Let $h=(2+\eta[-1])$. As a set, the homogeneous spectrum is
\[\spec^h(\kmw*)=A\amalg B\amalg C\amalg D\amalg \{([F^\times], 2),\; ([F^\times],\eta),\; ([F^\times],2,\eta)\}\]
where 
\begin{align*}
A&=\{([P_\alpha], h): \alpha \mbox{ an ordering on }F\},\\
 B&=\{([P_\alpha], 2, \eta): \alpha\mbox{ an ordering} \},\\ 
C&=\{([P_\alpha], h, p): \alpha\mbox{ an ordering}, p\mbox{ an odd prime}\}, \mbox{ and}\\
D&=\{([F^{\times}], \eta, p): p\mbox{ an odd prime}\}.
\end{align*} Topologically, $A\cong B\cong X_F$, where $X_F$ is the space of orderings on $F$ with the Harrison topology, that is the topology induced by the product topology on $\{\pm 1\}^{F^\times}$.
\end{thm}

\section{Homogeneous Ideals of $\kmw*$}\label{sec:spec}

We determine the homogeneous spectrum of $\kmw*$ (as a topological space) in this section. From Morel's pullback presentation, $\kmw0\cong W(F)\times_{Gr_I} \Z\cong GW(F)$, where $GW(F)$ is the Grothendieck-Witt ring of $F$. The image of the hyperbolic plane under this isomorphism is $h:=(2+\eta[-1])$. The Witt relation then says $h\eta=0$. This suggests that computing $\spec^h(\kmw*)$ amounts to computing the ideals in $\kmw*/(h)$ and $\kmw*/(\eta)$. Indeed, this is literally the case if our interest is only in the spectra qua sets.\footnote{Secretly, each of these quotients is one of the factors in the pullback presentation of \textsection 1. The spectrum of $\kmw*$ is then the pushout of the spectra of the quotients. We will do a little extra computation to avoid proving the general result about $\epsilon$-commutative graded rings. See \cite{schwede} for the commutative non-graded case.} We will start with these computations and then stitch the results together.

\subsection{Determination of $\spec^h(\kmw*/(h))$} \label{subsec:kmw/h}

We will need a more concrete presentation of $\kmw*/(h)$. A lemma of Morel supplies one in terms of the ring $L^*$ (described in \textsection 1).

\begin{lem}[Morel, \cite{Mo}]\label{lem:lstar} The quotient $\kmw*/(h)$ is isomorphic to $ L^*$ via the the map $[a]\mapsto \ip{a,-1}\eta\inv.$ 
\end{lem}

As with the Witt ring (see \cite{LL}), the spectrum of $L^*$ is deducible from the structure of the space, $X_F$, of orderings on $F$ and the following computation of the spectrum for real closed fields.

\begin{thm}
If $F$ is real closed, $\spec^h(L^*)=\{(p):p\mbox{ an odd prime or }0\}\amalg\{(\eta,2),(2,L^1),(\eta,2,L^1)\}$.
\end{thm}
\begin{proof}
In this case, $W(F)\cong \Z$, where the isomorphism carries $I$ to $2\Z$. So, letting $y$ be a generator for $L^1$, $L^*=\Z[\eta,y]/(\eta y-2)$. Let $J\in \spec^h(L^*)$, and let $p=\chr(J)$. (Recall that the characteristic of a prime ideal is the characteristic of its quotient ring.)

If $p=0$, then we want to show that $J$ contains no elements of degree 0, except $0$. If any element of nonzero degree were in $J$, say $ay^n$ or $a\eta^n$, then we would have $a2^n\in J$ in degree $0$. So, $J=(0)$.

If $p$ is an odd prime, then $J$ descends to a homogeneous prime ideal in the integral domain $L^*/(p)$. Every homogeneous element of $L^*/(p)$ is invertible, meaning its only homogeneous prime ideal is $(0)$. So $J=(p)$.

If $p$ is $2$, then $J$ contains either $\eta$ or $y$ or both. Clearly, $L^*/(2,\eta)\cong L^*/(2,y)\cong (\Z/2\Z)[x]$, where $x$ has degree $-1$ or $1$, respectively. The only homogeneous prime ideals here are $(x)$ and $(0)$, which pull back to $(2,\eta,y)$ and one of $(2,\eta)$ or $(2,y)$, respectively.
\end{proof}

Since the ring will show up again, we will write
\[R:=\Z[\eta,y]/(\eta y-2).\]

We will also make use of the following definitions.
\begin{dfn}
For any fields $F,K$, and any ring morphism $\phi: W(F)\rightarrow W(K)$, the canonical homogeneous extension $\phi^+:L^*F\rightarrow L^*K$ is given by:
\[\phi^+(q\eta^n)=\phi(q)\eta^n\]

We will want some notation for the the specific case of the signature maps.\footnote{Recall that the signature of a form $q=\ip{a_1,...,a_n}$ at the order $\alpha$ is $\sgn_\alpha(q)=\sum_{1\leq i\leq n} \alpha(a_i)$.} For $\alpha\in X_F$, let $F_\alpha$ be the real closure of $F$ at $\alpha$ and consider $\sgn_\alpha:W(F)\rightarrow W(F_\alpha)\cong \Z$ for $\alpha\in X_F$. We will write $J_\alpha^+$ for $\ker(\sgn_\alpha^+)$.
\end{dfn}

The next two theorems completely list the elements of $\spec^h(L^*)$.

\begin{thm}\label{orderidealsinl}
The following are homogeneous prime ideals in $L^*$:
\[J^+_\alpha,\;(J^+_\alpha, p),\; (J^+_\alpha,2,\eta),\;(L^1),\;(L^1,\eta)\mbox{ for }\alpha\in X_F\mbox{ and } \;p\mbox{ an odd prime}.\]
Furthermore, any homogeneous ideal containing some $J_\alpha^+$ is one of these, and if $\alpha\not=\beta$, $J^+_\alpha\not=J_\beta^+,\;(J^+_\alpha, p)\not=(J^+_\beta),\;$and $(J^+_\alpha,2,\eta)\not=(J^+_\beta,2,\eta)$.
\end{thm}
\begin{proof}
We know that $L^*/J^+_\alpha\cong R$, and $R$ has homogeneous spectrum 
\[\{(p):p\mbox{ an odd prime or }0\}\amalg\{(\eta,2),(2,L^1),(\eta,2,L^1)\}.\] The preimage of a spectrum under a quotient map can be computed in the $\epsilon$-graded commutative in the same way as in the commutative case (cf. \cite[Vol. 2. Ch VII, \textsection 2, Lemma 1]{ZS} for the graded case). Clearly, $(0),$ $(p)$, $(\eta, 2)$ pull back to $J^+_\alpha$, $(J_\alpha^+, p)$, and $(J_\alpha^+, 2, \eta)$ respectively. Also, since $\alpha(a)=\pm 1$, $\sum_{i=1}^n \alpha(a_i)=0$ implies $n$ must be even. so $J_\alpha^+\subseteq (I)=(L^1)$. Thus $(2, L^1)$ and $(2,\eta, L^1)$ pull back to $(L^1)$ and $(L^1,\eta)$. So,  $\spec^h(R)$ pulls back to the ideals listed. 

If $\beta\not=\alpha$, then there is some $a$ such that $\sgn_\alpha^+(\ip{1,a})=2\eta\inv$ and $\sgn_\beta^+(\ip{1,a})=0$ (namely, any element of $F^{\times}$ on which $\alpha$ and $\beta$ disagree), and $2\eta\inv$ is not in any of $J^+_\alpha,\;(J^+_\alpha, p),\;$ or $ (J^+_\alpha,2,\eta)$ as otherwise each of these would contain every element of positive degree in $L^*$.
\end{proof}

The next theorem implies there are no other ideals than the ones listed in Theorem \ref{orderidealsinl}. There is a much simpler proof for the non-two characteristic part of the spectrum, but we will present only the most general proof below.

\begin{thm}
Every element of $\spec^h(L^*)$ contains some $J^+_\alpha$.
\end{thm}
\begin{proof}
Suppose $J\in \spec^h(L^*)$. If $L^1\subseteq J$, then for any $\alpha\in X_F$, again since $J^+_\alpha\subseteq (L^1)$, $J^+_\alpha\subseteq J$. So, we may assume $L^1\not\subseteq J$.

For every $a\in F^\times$, in $L^*$,
\[(\ip{1,a}\eta\inv)(\ip{1,-a}\eta\inv)=0\]
so
\[\ip{1,a}\eta\inv\in J,\mbox{ or }\ip{1,-a}\eta\inv\in J\]
Thus, $\ip{a,b}\eta\inv\equiv \ip{\epsilon_0,\epsilon_1}\eta\inv$ where $\epsilon_i\in\{\pm1\}$. So, $\ip{a,b}\eta\inv=\pm 2 \eta\inv$ or $0$ modulo $J$.

Note that, since $L^1\not\subseteq J$ and $L^1$ is generated by elements of the form $\ip{a,b}\eta\inv$, $2\eta\inv\not\in J$. It follows that $\ip{-a,\alpha(a)}\eta\inv\in J$ characterizes a unique function $\alpha\in\{\pm1\}^{F^\times}$. We claim that $\alpha$ is an order and $J^+_\alpha\subseteq J$.

For the first claim, clearly $\alpha(1)=1$. Let $a,b$ be such that $\alpha(a)=\alpha(b)=1$. It suffices to show that $\alpha(-a)=-1$ and $\alpha(ab)=\alpha(a+b)=1$. First note:
\[\ip{-1,-(-a)}\eta\inv=-\ip{1,-a}\in J\]
so $\alpha(-a)=-1$. Consider the product:
\[0\equiv\ip{1,-a}\ip{1,-b}\eta\inv\equiv\ip{1,-a,-b,ab}\eta\inv\equiv\ip{-b,ab}\eta\inv\equiv\ip{-1,ab}\eta\inv \;\mbox{( mod $J$)}\]
so $\alpha(ab)=1$. We then have, working modulo $J$ and using the classical identity $\ip{a,b}\cong \ip{a+b,ab(a+b)}$:
\begin{align*} (2\eta\inv)(\ip{1,a+b}\eta\inv)& \equiv (\ip{1,1}\eta\inv)(\ip{1,a+b}\eta\inv)&\;\\
				\;&\equiv(\ip{1,ab}\eta\inv)(\ip{1,a+b}\eta\inv)&(\mbox{since }\alpha(ab)=1)\\
                                           \;&\equiv\ip{1,ab,(a+b), (a+b)ab}\eta^{-2}&\;\\
                                           \;&\equiv\ip{1,ab,a,b}\eta^{-2}&(\mbox{by the identity})\\
                                           \;&\equiv (2\eta\inv)(2\eta\inv).
\end{align*}

Since $2\eta\inv\not\in J$, we get $\ip{1,a+b}\eta\inv\not\in J$. Thus $\alpha(a+b)=1$.

For the second claim, note that for all $a_i\in F^\times$,
 \[\sgn_\alpha^+\left(\left(\sum \ip{a_i}\right)\eta^n \right)=\left(\sum\ip{\alpha(a_i)}\right)\eta^n\equiv \left(\sum\ip{a_i}\right)\eta^n\mbox{( mod $J$)}.\]
\end{proof}

It is straightforward to move the homogeneous spectrum of $L^*$ through the isomorphism of Lemma \ref{lem:lstar} ($[a]\mapsto\ip{a,-1}\eta\inv$) to get the spectrum of $\kmw*/(h)$.

\begin{cor}
If $J\in\spec^h(\kmw*/(h))$, $J$ is exactly one of the following:
\begin{itemize}
\item $([P_\alpha])$ for some $\alpha\in X_F$,
\item $([P_\alpha], p)$ for some $\alpha\in X_F$, odd prime $p$,
\item $([P_\alpha], 2,\eta)$ for some $\alpha\in X_F$,
\item $([F^\times])$, or
\item $([F^\times],\eta)$.
\end{itemize} 
\end{cor}

To determine the topological structure of the homogeneous spectrum of $L^*$ (with the topology induced by the Zariski topology), we will need a small lemma:

\begin{lem}
If $R$ is a graded ring, the sets $D(q)=\{J\in \spec(R): q\not\in J\}$ restricted to homogeneous $q$ form a subbasis for the Zariski topology on $\spec^h(R)$.
\end{lem}
\begin{proof} Consider $q=q_1+\cdots +q_n\in R$, where the $q_i$ are homogeneous. Then,
\begin{align*}
D(q)&=\{J\in \spec^h(R): q\not\in J\} \\
\; &=\{J\in \spec^h(R): \exists i\;(q_i\not\in J)\} \\
\; &=\bigcup_{1\leq i\leq n} \{J\in \spec^h(R): q_i\not\in J\} \\
\; &=\bigcup_{1\leq i \leq n} D(q_i).
\end{align*}
So the $D(q_i)$ generate the same topology as the $D(q)$. \end{proof}

Much of the topological information about $\spec^h(L^*)$ is coded by the topological structure of its minimal ideals. Let $\mspc$ denote the subspace of these minimal ideals.

\begin{thm} 
The minimum homogeneous spectrum $\mspc(L^*)$ is homeomorphic to $X_F$ with the Harrison topology.
\end{thm}
\begin{proof}
From previous computation, $\mspc(L^*)=\{J^\alpha_+:\alpha\in X_f\}$. We will show that the obvious bijection $\sigma_0$ (given by $\alpha\mapsto J_\alpha^+$) is a homeomorhism.

To see that $\sigma_0$ is continuous, consider the subbasic open set $D(q)$ in $\mspc(L^*)$, where $q$ is homogeneous. We have that $q=\tilde q\eta^n$ for some $\tilde q\in W(F)$.
\begin{align*}
\sigma_0\inv(D(q))&=\{\alpha: \sgn_\alpha^+(q)\not=0\} \\
\;&=\{\alpha: \sgn_\alpha(\tilde q)\not=0\} \\
\; &=\sgn(\tilde q)\inv (\Z\setminus \{0\}),
\end{align*}
where $\sgn(\tilde q):X_F\rightarrow \Z$ is the total signature given by $\sgn(\tilde q)(\alpha)=\alpha(\tilde q)$. Note that if $\tilde q=\ip{a_1,...,a_n}$, then $\sgn(\tilde q)=\sum_{1\leq i\leq n}\sgn(\ip{a_i})$. Giving $\Z$ the discrete topology, each $\sgn(\ip{a_i})$ is continuous by the definition of the Harrison topology, so $\sgn(\tilde q)$ is a sum of continuous functions and thus continuous. So, the set $\sgn(\tilde q)\inv (\Z\setminus\{0\})$ is open in $X_F$.

To see that $\sigma_0$ is open, consider the subbasic open set $H(a)=\{\alpha: \alpha(a)=1\}$ in $X_F$: 
\begin{align*}
\sigma_0(H(a))&=\{J_\alpha^+:\alpha(a)=1\} \\
\; &=\{J_\alpha^+: \sgn^+_\alpha(\ip{1,a})\not=0\} \\
\; &=D(\ip{1,a}),
\end{align*}
which is open.\end{proof}

A fuller description of the topology is subsumed by the description of the topology on $\spec^h(\kmw*)$.

\subsection{Determination of $\spec^h(\kmw*/(\eta))$} \label{subsec:kmw/eta}

Again, note that $\kmw*/(\eta)$ is isomorphic to $\km*$. The relevant arithmetical facts about this ring are established in Milnor's original paper \cite{Mi}. As with the quotient by $(h)$, the spectrum of the quotient by $(\eta)$ is described in terms of orderings. Efrat studies the relation between orderings and quotients of $\km*$ in \cite{Ef}, and his results are closely related to this computation.

We can immediately reduce the computation of the spectrum to the case where the ideals are of characteristic $2$.
\begin{lem}
\[\spec^h(\km*)\setminus \{J: 2\in J\}=\{(\km1, p): p\mbox{ an odd prime or }0\}\]
\end{lem}

\begin{proof}
Suppose $J\in \spec^h(\km*)$ and $\chr(J)\not=2$. Then, since $2[-1]=0$, $[-1]\in J$. We then have, for all $a\in F^\times$
\[[a,a]\equiv[a,-1]\equiv 0\; (\mbox{mod }J).\]
So $\km1\subset J$. Since $\km*/(\km1)\cong \Z$, $J$ is determined by its characteristic.
\end{proof}

For the homogeneous prime ideals in $\km*/(2)$, one could simply rely on the now resolved Milnor conjecture and move the characterization of ideals in $L^*$ through the quotient by $I$, but the story here is of independent interest.

\begin{thm}
Every element of $\spec^h(\km*/(2))$ is either $(\km1)$, or $([P_{\alpha}])$ for some order $\alpha$.
\end{thm}
\begin{proof}
Suppose $J\in \spec^h(\km*/(2))$ and $J\not=(\km1)$. We have that, for all $a\in F^\times$
\[[a,-a]=0\]
so, either $[a]$ or $[-a]=[-1]+[a]$ is in $J$. Thus, for some function $\alpha\in\{\pm 1\}^{F^\times}$,
\[[a]\equiv [\alpha(a)] \mbox{(mod $J$)}.\]
Since $J$ is assumed to not contain $\km1$, it follows that $[-1]\not\in J$, and $\alpha$ is onto. We will show that $\alpha$ is an ordering. First, $\alpha$ is multiplicative:
\[[ab]=[a]+[b]\equiv[\alpha(a)]+[\alpha(b)]=[\alpha(a)\alpha(b)]\mbox{(mod $J$)}.\]
And, $\alpha$ is additive. Suppose $\alpha(a)=\alpha(b)=1$:
\[[a+b][-ba\inv]=[1+ba\inv][-ba\inv]+[a][-ba\inv]\equiv 0\mbox{(mod $J$.)}\]
Since $[-1]\not\in J$, $[a+b]\in J$, and $\alpha(a+b)=1$.
\end{proof}

Efrat has studied ordering and prime ideals in $\km*$ using a kind of quotient construction $\km*/G$ where $G$ is a subgroup of $F^\times$ \cite{Ef}. It is worth noting that, in positive degrees, passing to the quotient by $([P_\alpha])$ is equivalent to passing to $\km*/S$, where $S$ is the subgroup of $\alpha$-positive elements. (In degree $0$, Efrat's construction is $\Z$, where our quotient is $\Z/2\Z$.)

The topology of the minimum spectrum is reducible to the Harrison topology.

\begin{thm}
$\mspc(\km*)$ is homeomorphic to $X_F\amalg\{(\km1)\}$.
\end{thm}
\begin{proof}
 Let $\sigma_2:X_F\rightarrow \mspc(\km*)$ be the obvious injection, $\sigma_2(\alpha)=([P_\alpha],2)$. From previous computation, $\mspc(\km*)=\sigma_2(X_F)\sqcup \{(\km1)\}$. We have that $D([-1])=\sigma_2(X_F)$ and $D(2)=\{(\km1)\}$. So, it suffices to show that $\sigma_2$ is continuous and open.

First, $\sigma_2$ is open:
\begin{align*}
\sigma_2(H(a))&=\{\sigma_2(\alpha): \alpha(a)=1\} \\
\; &=\{\sigma_2(\alpha): \alpha(-a)=-1\} \\
\;&=\{\sigma_2(\alpha): [-a]\not\in \sigma_2(\alpha)\} \\
\;&=D([-a])
\end{align*}

And, $\sigma_2$ is continuous. Consider some homogeneous $q=\sum_{j} [a_{j1},...,a_{jn}]$. For any $\alpha\in X_F$, we have that
\[q\equiv m_{\alpha}[-1]^n \;(\mbox{mod }\sigma_2(\alpha))\]
where $m_\alpha=|\{j: \forall i\; (\alpha(a_{ji})=-1)\}|$. We also have
\[\sgn_\alpha\left(\sum_j\pf{-a_{j1},...,-a_{jn}}\right)=m_\alpha 2^n,\]
where $\ip{\ip{-a_{j1},...,-a_{jn}}}=\prod_{1\leq i\leq n} \ip{1,-a_{ji}}$ is the Pfister form\footnote{According to at least one convention for Pfister forms.} associated to $-a_{j1},...,-a_{jn}$. Putting these together, we get
\begin{align*}
\sigma_2\inv(D(q))&=\{\alpha: 2\not| m_\alpha\}\\
\;&=\sgn\left(\sum_j\pf{-a_{j1},...,-a_{jn}}\right)\inv(2^n\Z\setminus 2^{n+1}\Z)
\end{align*}
which is open.
\end{proof}

\subsection{Determination of $\spec^h(\kmw*)$} \label{subsec:spec}

All that remains is to piece the spectra of the quotients together.

\begin{thm}
Let $J\in \spec^h(\kmw*)$. Then, $J$ is exactly one of the following:
\begin{enumerate}
\item $([P_\alpha],h,p)$ for some $\alpha\in X_F, p$ an odd prime,
\item $([P_\alpha],h)$ for some $\alpha\in X_F$,
\item $([P_\alpha],2,\eta)$ for some $\alpha\in X_F$,
\item $([F^\times], p, \eta)$ for some $p$ an odd prime,
\item $([F^\times], \eta)$,
\item $([F^\times], 2)$, or
\item $([F^\times], 2, \eta)$.
\end{enumerate}
\end{thm}

\begin{proof}
Note that these ideals are exactly those which arise by pulling the spectra of the quotients back along the quotient maps. An ideal in $\kmw*/(\eta)$ contains $h$ if and only if it contains 2, so if it is one of
\[([F^\times],2)\mbox{ or } ([P_\alpha],2)\mbox{ for some }\alpha\in X_F.\]

These pull back to $([F^\times], 2,\eta)$ and $([P_\alpha], 2, \eta)$, which descend in $\kmw*/(h)$ to
\[([F^\times], \eta)\mbox{ or }([P_\alpha], 2, \eta).\]

\end{proof}

The Hasse diagram for the inclusion poset of these seven types of ideals is given below.
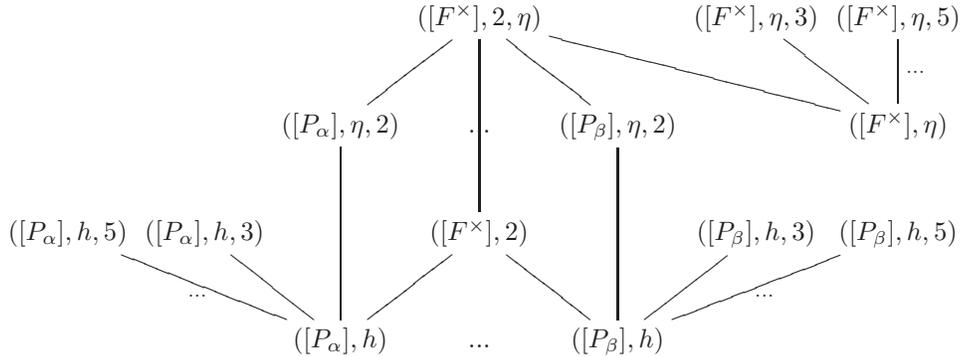
\[
\xymatrixcolsep{.05pc}\xymatrix{
 & & & ([F^\times],2,\eta)  \ar@{-}[ld] \ar@{-}[dd] \ar@{-}[rd] \ar@{-}[rrrd] & &([F^\times],\eta,3) \ar@{-}[rd] &([F^\times],\eta,5) \ar@{-}[d]^{...} \\
 & &([P_\alpha],\eta,2)\ar@{-}[dd] & ... &([P_\beta],\eta,2)\ar@{-}[dd]& &([F^\times],\eta)\\
([P_\alpha],h,5)\ar@{-}[rrd]_{...} &([P_\alpha],h,3)\ar@{-}[rd] & &([F^\times],2) \ar@{-}[ld] \ar@{-}[rd] & &([P_\beta],h,3)\ar@{-}[ld] &([P_\beta],h,5)\ar@{-}[lld]^{...}\\
 & & ([P_\alpha],h)& ... &([P_\beta],h)}
\]

All that is left to prove Theorem \ref{maintheorem} is to note that the inclusion maps 
\[\spec^h(\kmw*/(h))\rightarrow \spec^h(\kmw*)\] and 
\[\spec^h(\kmw*/(\eta))\rightarrow\spec^h(\kmw*)\] are homeomorphisms onto their images. Topologically, then, the spectrum of $\kmw*$ is an $X_F$-like spray of copies of the rational primes, with the prime 2 tripled, all glued together at two of the copies of 2, along with another copy of the rational primes at the ``top" of the diagram, with the closure of a set given by its upward closure in the inclusion poset pictured above.

\printbibliography

\end{document}